\documentclass{article}
\usepackage{maa-monthly}

\theoremstyle{theorem}
\newtheorem{theorem}{Theorem}

\theoremstyle{definition}

\begin{document}

\title{Formula of Volume of Revolution with Integration by Parts and Extension}
\markright{Volume of Revolution and Extension}
\author{Yi Liu , Jingwei Liu}

\maketitle

\begin{affil}
School of Mathematics and Statistics, Beijing Institute of Technology ,Beijing 100081,P.R China\\
Email : liuy0605@126.com.
\end{affil}
\begin{affil}
School of Mathematics and System Sciences, Beihang University, Beijing, 100083,P.R China\\
Email : liujingwei03@tsinghua.org.cn.
\end{affil}

\begin{abstract}
A calculation formula of volume of revolution with integration by parts of definite integral is derived based on monotone
function, and extended to a general case that curved trapezoids is determined by continuous, piecewise strictly monotone and differential
function. And, two examples are given, ones curvilinear trapezoids is determined by Kepler equation, and the other curvilinear trapezoids is a function transmuted from Kepler equation.
\end{abstract}

\noindent
In [1] (p494,ex71), if $f(x)$ is one-to-one and therefore has an inverse function $x=g(y)$, where $x \in [a,b]$,$a>0$ and $y\in [c,d]$. Let S be
the solid obtained by rotating about the y-axis the region bounded by $y =f(x)$ [where $f(x)\ge 0$], $y=0$, $x=a$ and $x=b$ (which is also called
curvilinear trapezoid [2]). Using cylindrical shells and slicing method, the volume of the given solid is
\begin{equation}
\label{eq1}
V=\int_a^b {2\pi xf(x)dx}=\pi b^2d - \pi a^2c - \int_c^d {\pi [g(y)]^2dy}
\end{equation}

This formula can also be proved by integration by parts with the increasing  case of $f$. However, there is slight divarication in formula (1).
\begin{enumerate}
\item[i)] If we guarantee the condition $y \in [c,d]$, formula (1) does not hold in case of decreasing $f$, as under this condition $c=f(b)$, $d=f(a)$.
\item[ii)] If we ensure that the condition $c=f(a)$, $d=f(b)$ is established regardless of increasing or decreasing of $f$, the condition $y\in [c,d]$ does not hold in case of decreasing of $f$, as $c>d$. However, formula (1) still holds in mathematical style.
\end{enumerate}

\section{Volume of Revolution and Extension.}

A reasonable expression would be as follows.
\begin{theorem} Let $y=f(x)$ be nonnegative, continuous, differential and strictly monotone on [a,b] ($a>0$), it's inverse function is $x=g(y)$, $y
\in [c,d]$. Then,
\begin{equation}
\label{eq2}
\pi\int_c^d{[g(y)]^2dy}=\mbox{sgn}(f(b)-f(a))\left\{{\pi[{b^2f(b)-a^2f(a)}]-2\pi\int_a^b {xf(x)dx}}\right\}
\end{equation}
\begin{equation}
\label{eq3}
\pi\int_a^b{[f(x)]^2dx}=\mbox{sgn}(g(d)-g(c))\left\{{\pi[{d^2g(d)-c^2g(c)}]-2\pi\int_c^d {yg(y)dy}}\right\}
\end{equation}
where $y=\mbox{sgn}(x)$ is the sign function.
\end{theorem}

\begin{proof}

(I) Let $y=f(x)$ be strictly monotone increasing(fig1),then $c=f(a)$, $d=f(b)$, $c<d$, and
\begin{equation} \label{eq4}
\begin{array}{ll}
 V_Y &=\displaystyle \pi \int_c^d {[g(y)]^2dy}=\pi \int_a^b {x^2df(x)}=\pi \left[{x^2f(x)\vert_a^b - \int_a^b {f(x)dx^2}}\right] \\
     &=\displaystyle \pi \left[ {b^2f(b) -a^2f(a)} \right] - 2\pi \int_a^b {xf(x)dx} \\
 \end{array}
\end{equation}

\begin{figure}[!htbp]
\begin{center}
\begin{minipage}{6.2cm}
\includegraphics[width=6cm,height=4cm]{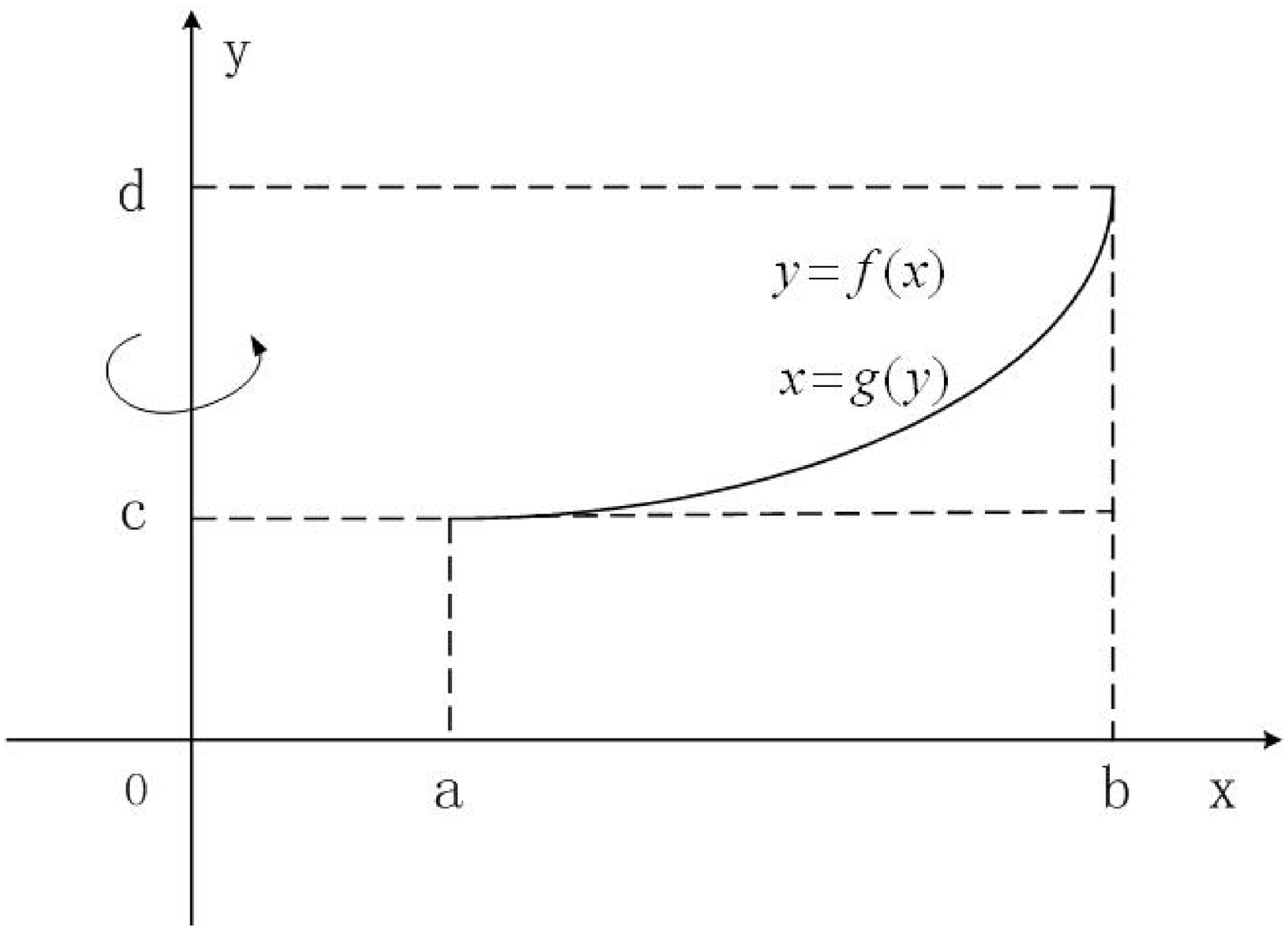}
\label{fig1}
\caption{$y=f(x)$ is strictly monotone increasing.}
\end{minipage}
\begin{minipage}{6.2cm}
\includegraphics[width=6cm,height=4cm]{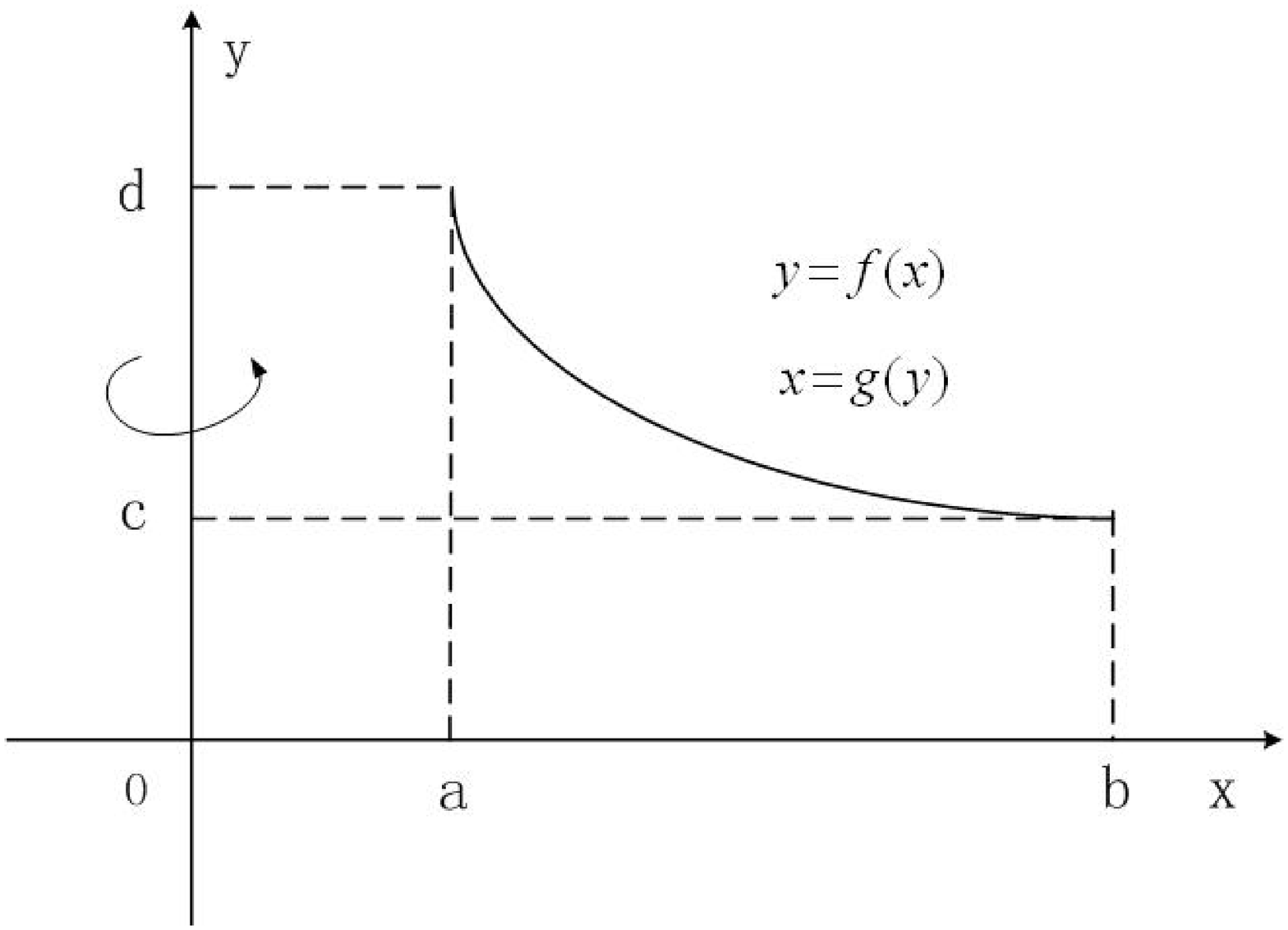}
\label{fig2}
\caption{$y=f(x)$ is strictly monotone decreasing.}
\end{minipage}
\end{center}
\end{figure}

(II) Let $y=f(x)$ be strictly monotone decreasing (fig 2), then $d =f(a)$, $c=f(b)$, ($c<d$), and
\begin{equation} \label{eq5}
\begin{array}{ll}
 V_Y &=\displaystyle \pi\int_c^d {[g(y)]^2dy}=\pi \int_b^a {x^2df(x)}=-\pi \int_a^b{x^2df(x)}\\
     &=\displaystyle -\pi\left[{x^2f(x)\vert_a^b -\int_a^b {f(x)dx^2}} \right]\\
    &=\displaystyle  2\pi \int_a^b {xf(x)dx} - \pi \left[{b^2f(b) - a^2f(a)} \right] \\
 \end{array}
\end{equation}

Similarly, we can prove the other two cases of volume of solids of revolution with integration by parts as follows.

(III) If $x=g(y)$ is strictly monotone increasing (fig 1), then $a=g(c)$, $b=g(d)$, ($a<b$), and
\begin{equation} \label{eq6}
\begin{array}{ll}
 V_X &=\displaystyle \pi \int_a^b {[f(x)]^2dx}=\pi \int_c^d {y^2dg(y)}=\pi \left[{y^2g(y)\vert_c^d-\int_c^d {g(y)dy^2}} \right] \\
    &=\displaystyle \pi \left[ {d^2g(d) -c^2g(c)} \right]-2\pi \int_c^d {yg(y)dy} \\
 \end{array}
\end{equation}

(IV) If $x=g(y)$ is strictly monotone decreasing (fig 2), then $b=g(c)$, $a=g(d)$, ($a<b$), and
\begin{equation} \label{eq7}
\begin{array}{ll}
 V_X &=\displaystyle \pi \int_a^b {[f(x)]^2dx}=\pi \int_d^c {y^2dg(y)}=- \pi \int_c^d {y^2dg(y)}\\
     &=\displaystyle -\pi \left[ {y^2g(y)\vert_c^d -\int_c^d {g(y)dy^2}}\right]   \\
     &=\displaystyle 2\pi \int_c^d {yg(y)dy} - \pi \left[{d^2g(d) - c^2g(c)} \right] \\
 \end{array}
\end{equation}

Hence, this ends the proof.
\end{proof}

If $y=f(x)$ is continuous and not strictly monotone on $[a,b]$, and $[a,b]$ can be divided into several strictly monotone intervals of $f(x)$, we have the
following lemma.

\textbf{Lemma 1.} Suppose that $y=f(x)$is continuous and piecewise strictly monotone on $[a,b]$, if $f(a)\ne f(b)$, $\min \{f(a),f(b)\}<f(x)<
\max\{f(a),f(b)\}$, $x\in (a,b)$, and $f(x)$ is not strictly monotone on $[a,b]$, the number of local extreme value points in $(a,b)$ is even.

\begin{proof} We prove the lemma in two cases, $f(a)<f(b)$, and $f(a)>f(b)$.

In case of $f(a)<f(b)$, since $f(x)>f(a)$, $x\in (a,b)$, the nearest extreme value point to $x=a$ is a local maximum value point.
Since $f(x)<f(b)$, $x\in (a,b)$, the nearest extreme value point to $x=b$ is a local minimum value point. As $y=f(x)$ is continuous and piecewise strictly monotone on $[a,b]$, the number of extreme value points should be even. OR, using method of reduction to absurdity, if the number of extreme values was odd, the last extreme value point (nearest extreme value point to $x=b$) might be a local maximum extreme value point, it is a contradiction.

Similarly, we can prove the other case. Hence, we complete the proof.
\end{proof}

Then, the volume of curvilinear trapezoidal composed by curved edge $y=f(x)$ in term of y-axis, and rotating with y-axis can reach the following conclusion.

\begin{theorem} Let $y=f(x)$ be continuous, piecewise strictly monotone and differential on $[a,b]$. $\Gamma $is the curve determined by $y=f(x)$ on $[a,b]$. Denote $c=\min (f(a),f(b))$, $d=\max (f(a),f(b))$. Suppose that both $y=c$ and $y=d$ $(c<d)$ intersect $y=f(x)$ in only one point
on $[a,b]$ respectively. Then the volume of the plane graph along $x=0$, $y=c$, $\Gamma$, $y=d$ rotating with $y$-axis is
\begin{equation} \label{eq8}
V_Y=\mbox{sgn}(f(b) - f(a))\left\{ {\pi \left[ {b^2f(b) - a^2f(a)} \right] - 2\pi\int_a^b {xf(x)dx}}\right\}
\end{equation}
\end{theorem}

\begin{proof} According to Lemma 1, since $y=f(x)$ is continuous, and strictly piecewise monotone on $[a,b]$, assume that the strictly piecewise
monotone intervals of $y=f(x)$ are $[x_i,x_{i+1}]$, $i=0,1,\cdots,n$, $x_0=a$, $x_{n+1}=b$, $i=1,2,\cdots,n$, and $[a,b]=\bigcup\limits_{i= 0}^n {[x_i ,x_{i + 1} ]} $, where $x_1 , \cdots ,x_n$ are extreme value points of $f$.

\begin{figure}[!htbp]
\begin{center}
\begin{minipage}{6cm}
\includegraphics[width=6cm,height=4cm]{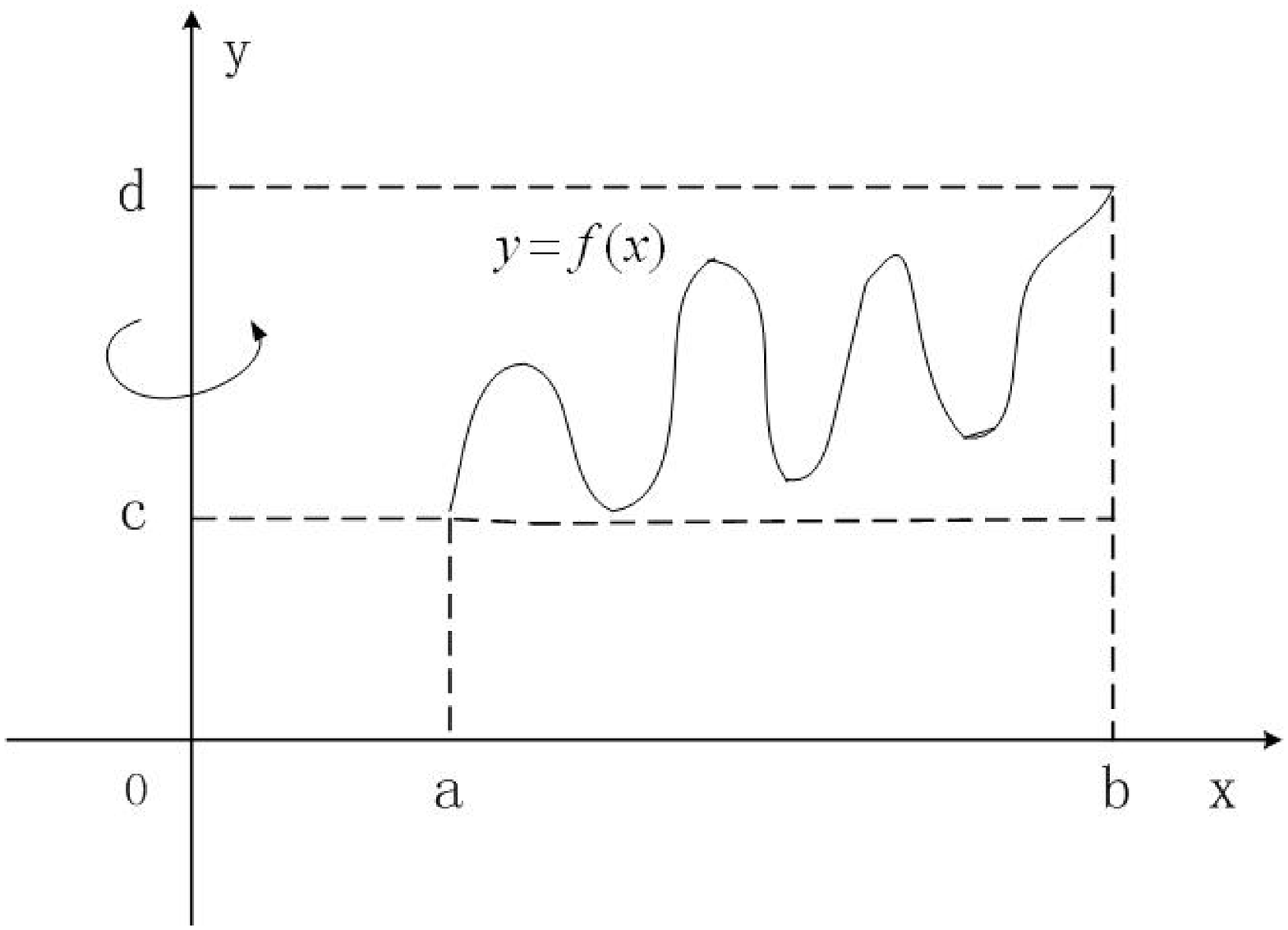}
\label{fig3}
\caption{$f$ is continuous and piecewise strictly monotone with $f(b)>f(a)$.}
\end{minipage}
\quad
\begin{minipage}{6cm}
\includegraphics[width=6cm,height=4cm]{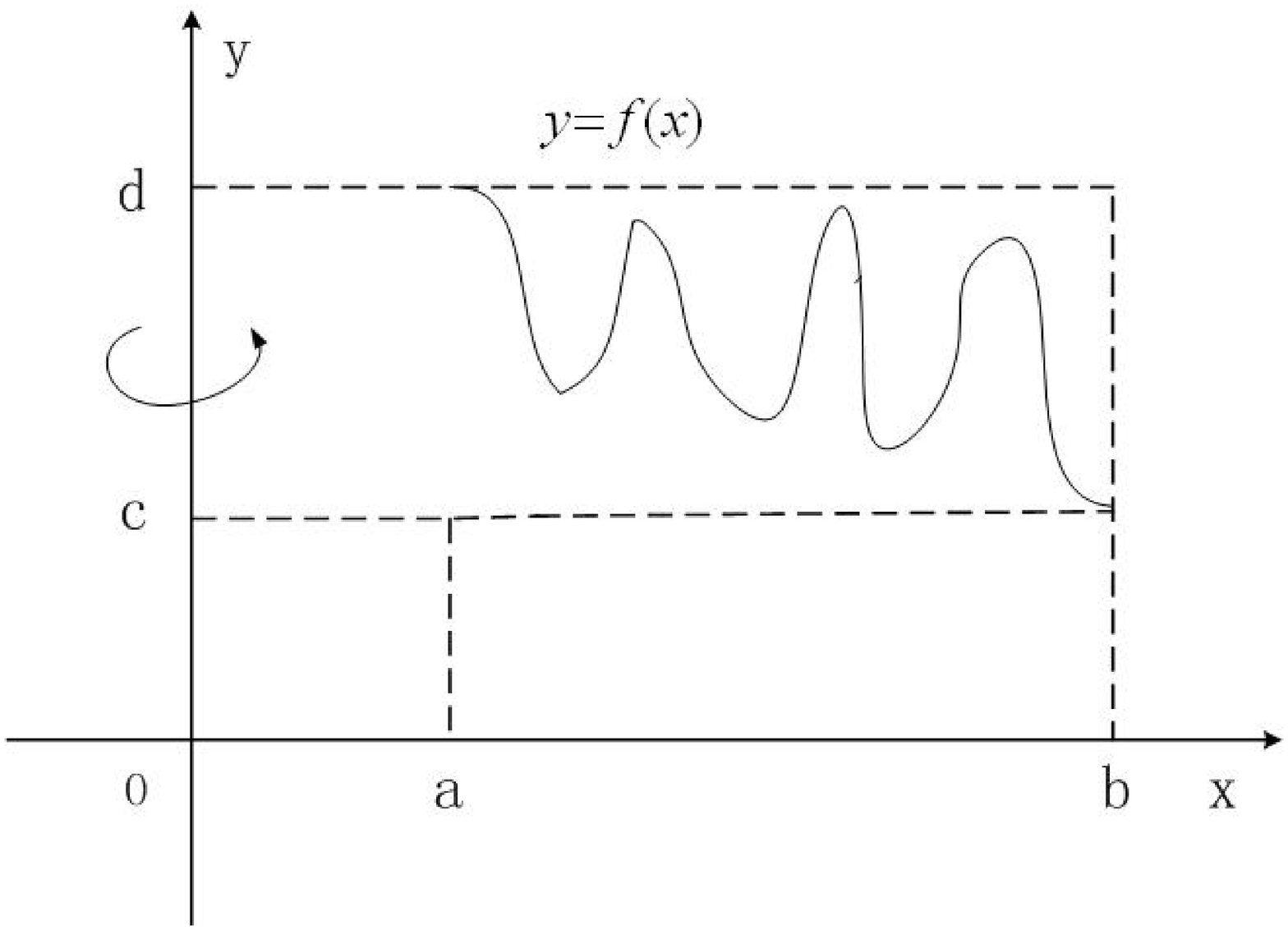}
\label{fig4}
\caption{$f$ is continuous and piecewise strictly monotone with $f(b)<f(a)$.}
\end{minipage}
\end{center}
\end{figure}

(1) If $f(b)>f(a)$ (fig3), as both $y=c$ and $y=d$ $(c<d)$ intersect $y= f(x)$ only once on $[a,b]$ respectively, $y=f(x)$ is strictly
increasing on both $[x_0,x_1]$ and $[x_n,x_{n+1} ]$, where $n$ is even (if $n=0$, Theorem 2 degrades to Theorem 1) and $c=f(a)$, $d=f(b)$.

Denote $\Gamma_i :y=f(x),x \in [x_i ,x_{i + 1} ]$, and $V_{\Gamma_i}$ is the volume of curved trapezoid $\Gamma_i $ with $y$ -axis rotating
with $y$-axis. Since $y=f(x)$ is strictly monotone and differential on $[x_i, x_{i+1}]$, Theorem 1 would be employed in the volume of solids of revolution on each curvilinear trapezoidal determined by piecewise strictly monotone and differential curve.
Then,
\begin{equation} \nonumber
\begin{array}{ll}
V_{\Gamma_{2k}} &=\displaystyle \pi [x_{2k+1}^2 f(x_{2k+1} )-x_{2k}^2 f(x_{2k} )]-2\pi \int_{x_{2k} }^{x_{2k + 1}}{xf(x)dx}\\
V_{\Gamma_{2k+1}} &=\displaystyle 2\pi \int_{x_{2k+1} }^{x_{2k+2}}{xf(x)dx}-\pi [x_{2k+2}^2 f(x_{2k+2} )-x_{2k+1}^2 f(x_{2k+1})] \\
                  & \quad  \quad \quad \quad \quad \quad \quad  \quad \quad \quad \quad \quad k=0,1,\cdots,n/2. \\
\end{array}
\end{equation}

Using mathematical induction, we can obtain,
\begin{equation} \label{eq9}
\begin{array}{ll}
 V_Y &=\displaystyle V_{\Gamma_0}- V_{\Gamma_1}+ V_{\Gamma_2}+ \cdots + (-1)^nV_{\Gamma_n }=\sum\limits_{i=0}^n {(- 1} )^{i - 1}V_{\Gamma_i}\\
   &=\displaystyle \sum\limits_{i= 0}^n {\left\{ {\pi [x_{i + 1}^2 f(x_{i + 1} ) - x_i^2 f(x_i )] - 2\pi\int_{x_i }^{x_{i + 1}}{xf(x)dx}}\right\}} \\
   &=\displaystyle \pi [x_{n +1}^2 f(x_{n+1} ) - x_0^2 f(x_0 )] - 2\pi \int_{x_0 }^{x_{n+1} }{xf(x)dx} \\
   &=\displaystyle \pi [b^2f(b) -a^2f(a)] - 2\pi \int_a^b {xf(x)dx} \\
 \end{array}
\end{equation}

(2) If $f(a)>f(b)$ (fig4), as both $y=c$ and $y=d$ $(c<d)$ intersect $y=f(x)$ one point on $[a,b]$ respectively. Then $y=f(x)$ is strictly
decreasing on $[x_0 ,x_1]$ and $[x_n,x_{n+1}]$, where $n$ is even, and $d=f(a)$, $c=f(b)$.

Denote $\Gamma_i :y=f(x),x\in [x_i,x_{i+1}]$,and $V_{\Gamma_i}$ is the volume of curvilinear trapezoid of $\Gamma_i $ in term of $y$-axis rotating
with $y$-axis. Then,
\begin{equation} \nonumber
\begin{array}{ll}
V_{\Gamma_{2k}}   &=\displaystyle 2\pi \int_{x_{2k} }^{x_{2k+1}}{xf(x)dx} - \pi[x_{2k+1}^2 f(x_{2k+1} ) - x_{2k}^2 f(x_{2k})]\\
V_{\Gamma_{2k+1}} &=\displaystyle \pi [x_{2k+2}^2 f(x_{2k+2})-x_{2k+1}^2 f(x_{2k+1})] -2\pi \int_{x_{2k+1} }^{x_{2k+2}}{xf(x)dx}\\
                  & \quad  \quad \quad \quad \quad \quad \quad  \quad \quad \quad \quad \quad k=0,1,\cdots,n/2. \\
\end{array}
\end{equation}

Using mathematical induction, we can obtain,
\begin{equation} \label{eq10}
\begin{array}{ll}
 V_Y &=\displaystyle V_{\Gamma_0}- V_{\Gamma_1}+ V_{\Gamma_2}+ \cdots + (-1)^nV_{\Gamma_n }=\sum\limits_{i=0}^n {(- 1} )^{i - 1}V_{\Gamma_i }\\
   &=\displaystyle \sum\limits_{i= 0}^n {\left\{ {2\pi \int_{x_i }^{x_{i + 1}}{xf(x)dx} - \pi [x_{i + 1}^2f(x_{i + 1} ) - x_i^2 f(x_i )]} \right\}} \\
   &=\displaystyle 2\pi \int_{x_0}^{x_{n+1}}{xf(x)dx} - \pi [x_{n+1}^2 f(x_{n+1} ) - x_0^2 f(x_0 )]\\
   &=\displaystyle 2\pi \int_a^b {xf(x)dx} - \pi [b^2f(b) - a^2f(a)] \\
 \end{array}
\end{equation}

To sum up formulae(9)(10), the theorem is completely proved.
\end{proof}

Similarly, we can obtain the following theorem:

\begin{theorem} Let $x=g(y)$ be continuous, piecewise strictly monotone and differential on $[c,d]$. $\Gamma $ is the curve determined by $x
=g(y)$ on $[c,d]$. Denote $a=\min (g(c),g(d))$, $b=\max (g(c),g(d))$. Suppose that both $x=a$ and $x=b$ $(a<b)$ intersect $x=g(y)$ in only one point on
$[c,d]$ respectively. Then the volume of the plane graph along $y=0$, $x=a$, $\Gamma $ and $x=b$ rotating with $x$-axis is
\begin{equation} \label{eq11}
V_X=\mbox{sgn}(g(d) - g(c))\left\{ {\pi \left[ {d^2g(d) - c^2g(c)} \right] - 2\pi \int_c^d {yg(y)dy}}\right\}
\end{equation}
\end{theorem}

\section{Application.}

\textbf{Example 1} For Kepler equation [3]: $y=x + \varepsilon \sin y$, $\varepsilon \in (0,1)$, $y\in [0,2\pi]$. (fig 5)
(i) To calculate the volume of solids of revolution of curved trapezoid composed by Kepler curve, $y=0,y=2\pi$ and $y$-axis rotating with
 $y$-axis.
(ii) To calculate the volume of solids of revolution of curved trapezoid composed by Kepler curve, $x=0,x=2\pi$ and $x$-axis rotating
with $x$-axis.

\textbf{Solution}  (i) For $x=g(y)=y-\varepsilon \sin y$,$\varepsilon \in (0,1)$,
\begin{equation} \nonumber
\begin{array}{ll}
 V_Y &=\displaystyle \pi \int_0^{2\pi}{[g(y)]^2dy}=\pi \int_0^{2\pi}{[y -\varepsilon \sin y]^2dy} \\
   &=\displaystyle [\frac{y^3}{3} +2\varepsilon (y\cos y - \sin y) + \varepsilon ^2(\frac{y}{2} - \frac{\sin2y}{4})]\vert_0^{2\pi}\\
   &=\displaystyle \frac{8}{3}\pi ^3 + (4\varepsilon +\varepsilon ^2)\pi \\
 \end{array}
\end{equation}

(ii) Since $\displaystyle \frac{dx}{dy}=1-\varepsilon\cos y>0$, the curve determined by Kepler function $x=g(y)=y-\varepsilon \sin y$ is strictly increasing in term of $y$ on $[0,2\pi]$. According to inverse function theorem of monotone continuous function, it's inverse function $y=f(x)$ exists in term of $x$ on $[0,2\pi]$ and is strictly increasing. Then, the function $y=f(x)$ determined by Kepler function is an implicit function and also transcendental equation.

Using formula(3), $a=0$, $b=2\pi$, $c=0$, $d=2\pi$, we obtain
\begin{equation} \nonumber
\begin{array}{ll}
 V_X &=\displaystyle \pi \int_0^{2\pi}{[f(x)]^2dx}=8\pi ^4 - 2\pi \int_0^{2\pi}{yg(y)dy} \\
   &=\displaystyle 8\pi ^4 - 2\pi \int_0^{2\pi}{y(y - \varepsilon \sin y)dy} \\
   &=\displaystyle 8\pi ^4 - 2\pi [\frac{y^3}{3} + \varepsilon (y\cos y - \sin y)]\vert_0^{2\pi}\\
   &=\displaystyle \frac{8}{3}\pi ^4 + 4\varepsilon \pi ^2 \\
 \end{array}
\end{equation}

\textbf{Example 2} To calculate the volume of revolution of curvilinear trapezoid: $y=f(x)=\displaystyle \frac{x}{\pi}+ \sin x$ ($0 \le x \le 2\pi )$, $y=0$, $y=2\pi $ and $y$-axis rotating with $y$-axis.

\begin{figure}[!htbp]
\begin{center}
\begin{minipage}{6cm}
\includegraphics[width=6cm,height=4cm]{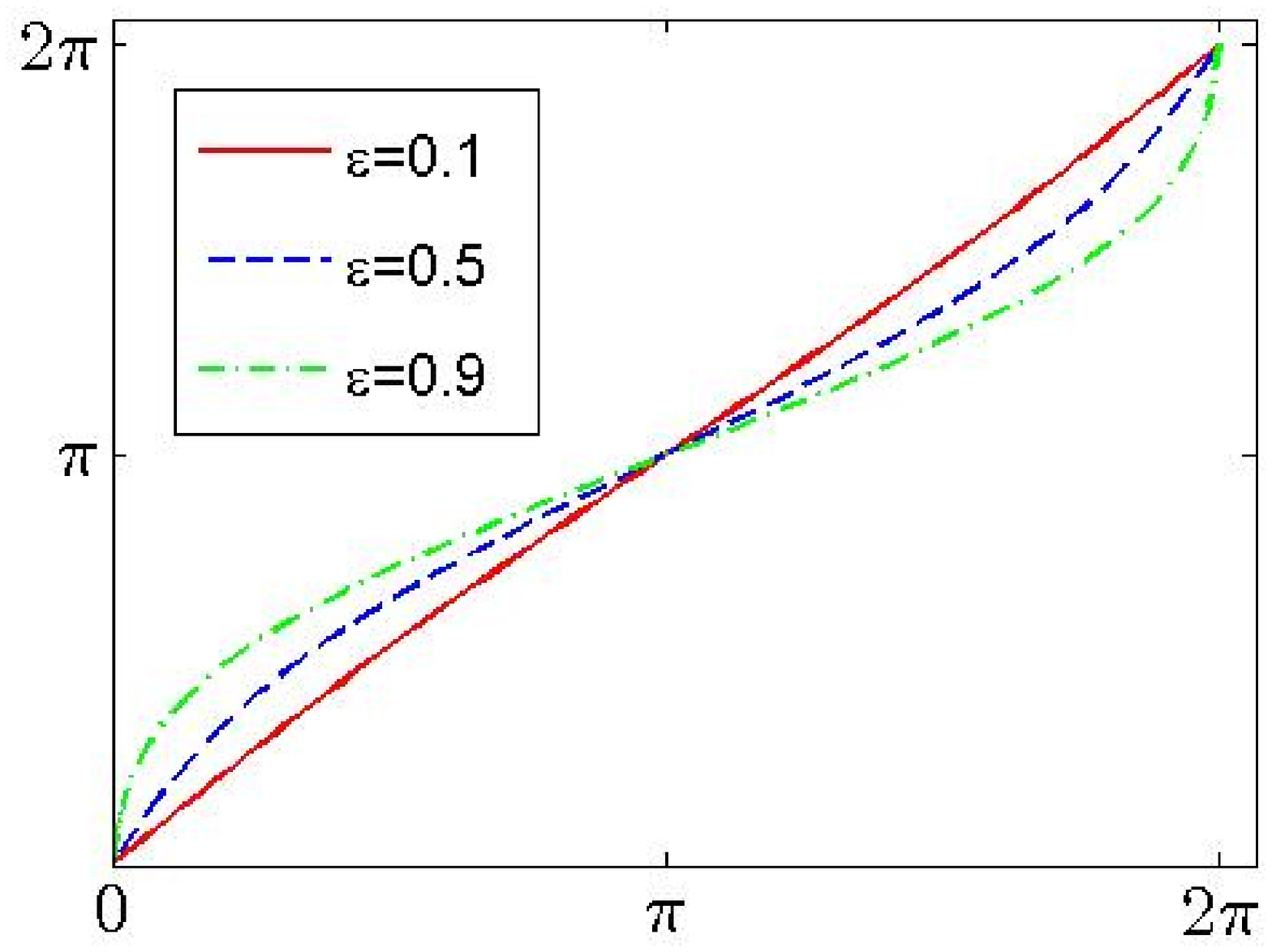}
\caption{Kepler equation $y=x + \varepsilon \sin y$, $\varepsilon \in (0,1)$.}
\end{minipage}
\quad
\begin{minipage}{6cm}
\includegraphics[width=6cm,height=4cm]{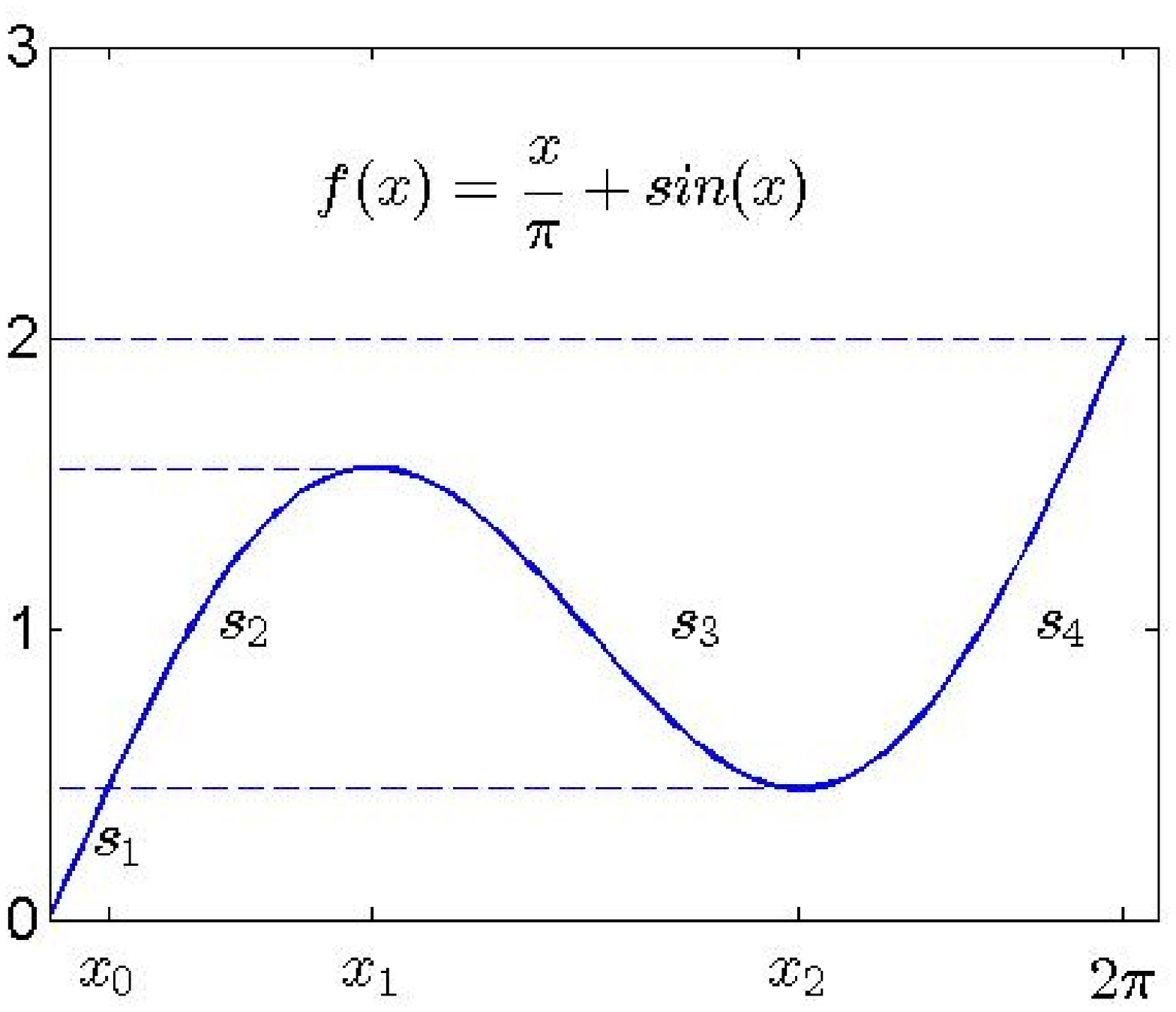}
\caption{Piecewise monotone intervals of $f(x) =\displaystyle \frac{x}{\pi}+ \sin x$,$0 \le x \le 2\pi $.}
\end{minipage}
\end{center}
\end{figure}

\textbf{Solution 1}: We give the classical solution.

(1) Let $f'(x) =\displaystyle (\frac{x}{\pi}+ \sin x)'=\displaystyle \frac{1}{\pi}+ \cos x=0$,
we obtain two extreme value points, $x_1 =\displaystyle \arccos (- \frac{1}{\pi })$ and $x_2 =\displaystyle  2\pi - \arccos (- \frac{1}{\pi})$.
And,
\[
f(x_1 ) =\displaystyle \frac{1}{\pi }\arccos (- \frac{1}{\pi }) + \sin (\arccos (-\frac{1}{\pi })),
\]
\[
f(x_2 ) =\displaystyle \frac{1}{\pi }[2\pi - \arccos (- \frac{1}{\pi })] - \sin (\arccos(- \frac{1}{\pi })).
\]

(2) Denote:
\[
\begin{array}{ll}
s_1 :y=f(x), 0 \le x \le x_0 .           &  s_2 :y=f(x), x_0 \le x \le x_1 .\\
s_1 + s_2 :y=f(x),    0 \le x  \le x_1 . & \\
s_3 :y=f(x), x_1 \le x \le x_2 .           & s_4 :y=f(x),    x_2 \le x \le x_3 .\\
\end{array}
\]

The volumes of curved trapezoids of $s_1, s_2, s_3, s_4, s_1+s_2$ in term of $y$-axis rotating with $y$-axis are denoted as $V_{s_1}$,$V_{s_2}$,$V_{s_3}$ ,$V_{s_4}$,$V_{s_1 + s_2}$,and the volume of curved trapezoid $y=f(x)$ ($0\le x \le 2\pi$), $y=0$, $y=2$, $x=0$ rotating
with $y$-axis is $V_Y$.Then,
\begin{equation} \label{eq12}
\begin{array}{ll}
V_Y &=\displaystyle V_{s_4}- (V_{s_3}- V_{s_2}) + V_{s_1 }=V_{s_4}- V_{s_3}+ (V_{s_2}+ V_{s_1}) \\
  &=\displaystyle V_{s_4}- V_{s_3}+ V_{s_1 + s_2 }\\
\end{array}
\end{equation}

According to Theorem 1,
\begin{equation} \nonumber
\begin{array}{ll}
 V_{s_1 + s_2 }&=\displaystyle \pi x_1 ^2f(x_1 )-0 - 2\pi \int_0^{x_1}{x(\frac{x}{\pi }+ \sin x)dx} \\
               &=\displaystyle \pi x_1^2f(x_1 ) - 2\pi [\frac{x^3}{3\pi}+ (\sin x - x\cos x)]\vert_0^{x_1}\\
               &=\displaystyle \pi x_1 ^2f(x_1 ) - 2\pi [\frac{x_1 ^3}{3\pi}+ (\sin x_1 - x_1 \cos x_1 )] \\
 \end{array}
\end{equation}
\begin{equation} \nonumber
\begin{array}{ll}
 V_{s_3}&=\displaystyle 2\pi \int_{x_1 }^{x_2}{x(\frac{x}{\pi}+ \sin x)dx} - \pi [x_2^2 f(x_2 ) - x_1^2 f(x_1 )] \\
          &=\displaystyle 2\pi [\frac{x^3}{3\pi}+(\sin x - x\cos x)]\vert_{x_1 }^{x_2}- \pi [x_2^2 f(x_2 ) - x_1^2 f(x_1)] \\
          &=\displaystyle 2\pi [\frac{x_2 ^3}{3\pi }+ (\sin x_2 - x_2 \cos x_2 )] - 2\pi [\frac{x_1 ^3}{3\pi}+ (\sin x_1 - x_1
            \cos x_1 )] - \pi [x_2^2 f(x_2 ) - x_1^2 f(x_1 )] \\
 \end{array}
\end{equation}
\begin{equation} \nonumber
\begin{array}{ll}
 V_{s_4} &=\displaystyle \pi [x_3^2 f(x_3 ) - x_2^2 f(x_2 )] - 2\pi \int_{x_2 }^{x_3}{x(\frac{x}{\pi}+ \sin x)dx} \\
          &=\displaystyle \pi [x_3^2 f(x_3 ) - x_2^2 f(x_2 )] -2\pi [\frac{x^3}{3\pi}+ (\sin x - x\cos x)]\vert_{x_2 }^{x_3}\\
          &=\displaystyle \pi [x_3^2 f(x_3 ) - x_2^2 f(x_2 )] - 2\pi [\frac{x_3 ^3}{3\pi}+ (\sin x_3 - x_3 \cos x_3 )] + 2\pi [\frac{x_2^3}{3\pi}+ (\sin x_2- x_2 \cos x_2 )] \\
 \end{array}
\end{equation}

Substituting $V_{s_1 + s_2}$,$V_{s_3}$,$V_{s_4}$,$f(x_3 )=f(2\pi )=2$  into formula (12), we obtain
\[
\begin{array}{ll}
 V_Y &=\displaystyle V_{s_4}- V_{s_3}+ V_{s_1 + s_2}=\displaystyle \pi x_3^2 f(x_3 ) - 2\pi[\frac{x_3 ^3}{3\pi}+ (\sin x_3 - x_3 \cos x_3 )] \\
   &=\displaystyle 8\pi ^3 - 2\pi[\frac{8\pi ^2}{3} - 2\pi ] =\displaystyle \frac{8\pi ^3}{3} + 4\pi ^2 \\
 \end{array}
\]

\textbf{Solution 2}:  Using formula (8), the volume of solids of revolution is
\[
\begin{array}{ll}
V_Y &=\displaystyle \mbox{sgn}(f(b) - f(a))\left\{ {\pi \left[ {b^2f(b) - a^2f(a)} \right] - 2\pi\int_a^b {xf(x)dx}}\right\}\\
  &=\displaystyle \left\{ {\pi \left[ {(2\pi )^2f(2\pi ) - 0} \right] - 2\pi \int_0^{2\pi }{x(\frac{x}{\pi}+ \sin x)dx}}\right\} \\
  &=\displaystyle 8\pi ^3 - 2\pi [\frac{x^3}{\pi}+ \sin x - x\cos x]\vert_0^{2\pi}\\
  &=\displaystyle 8\pi ^3 - 2\pi [\frac{8\pi ^2}{3} - 2\pi ]=\frac{8\pi ^3}{3} + 4\pi ^2 \\
 \end{array}
\]

Similar to Example 2, interchanging $x$ and $y$, we can obtain an example rotating with $x$-axis using Theorem 3.

\textbf{Example 3 }To calculate the volume of curved trapezoid of $ x =g(y)=\displaystyle \frac{y}{\pi}+ \sin y$
($0 \le y \le 2\pi )$, $x=0,x=2,y=0$ rotating with $x$-axis.

(The solution is similar to Example 2, we omit the detail. )

\begin{acknowledgment}{Acknowledgment.}
This paper is partially supported by 863 Project of China (2008AA02Z306), and Major Program of the National Natural Science Foundation of China (No.61327807).
\end{acknowledgment}

\vfill\eject

\end{document}